\documentclass[12pt]{article}
\usepackage[a4paper,margin=1in]{geometry}
\usepackage{amsmath,amssymb,amsthm,mathtools}
\usepackage{graphicx}
\usepackage{tikz,cite}
\usepackage{float}
\usepackage{booktabs}
\usepackage{array}
\usepackage{enumitem}
\usepackage{hyperref}
\hypersetup{colorlinks=true,linkcolor=blue,citecolor=blue,urlcolor=blue}
\renewenvironment{proof}{{\bf \noindent Proof.}}{\qed}
\everymath{\displaystyle}

\newtheorem{theorem}{Theorem}[section]
\newtheorem{lemma}[theorem]{Lemma}
\newtheorem{proposition}[theorem]{Proposition}
\newtheorem{corollary}[theorem]{Corollary}
\newtheorem{remark}[theorem]{Remark}
\newtheorem{example}[theorem]{Example}

\title{Metric Bases of Barycentric and Matching Subdivisions of Zero-Divisor Graphs}
\author{Vidya S$^{1}$, Prasanna Poojary$^{2,}\footnote{Corresponding author (poojary.prasanna@manipal.edu)}$, Bilal Ahmad Rather$^3$\\
Vadiraja Bhatta G R$^4$\\
\footnotesize $^{1,2,4}${\it Manipal Institute of Technology, Manipal Academy of Higher Education, Manipal, India;}\\
\footnotesize $^{3}${\it School of Mathematics and Statistics, Shandong University of Technology, Zibo 255049, China}\\
\footnotesize$^{3}${\it Centre for Research Impact \& Outcome, Chitkara University Institute of Engineering and Technology}\\
\footnotesize{\it Chitkara University, Rajpura, 140401, Punjab, India.}	\\
\texttt{vidyas1444@gmail.com; poojary.prasanna@manipal.edu;} \\ \texttt{bilalahmadrr@gmail.com; vadiraja.bhatta@manipal.edu}
}
\date{}

\begin{document}
\maketitle

\begin{abstract}
\noindent
In this paper, we study metric bases and related metric properties for barycentric and partial matching subdivisions of the zero-divisor graph of $\mathbb Z_{pq}$, where $p$ and $q$ are distinct odd primes with $q>p$. We first recall the natural partition of the zero-divisor graph into the two prime classes and then give a detailed characterization of those subsets of $BS(\Gamma(\mathbb Z_{pq}))$ that form metric bases when $q\geq 2p-1$. The proof is expanded by separating the role of closed neighborhoods, rows of subdivision vertices, and forbidden twin configurations. We then investigate $M$-subdivision graphs obtained by subdividing selected edges of $\Gamma(\mathbb Z_{pq})$. In addition to the lower bounds for subdivisions of $p-3$ and $p-2$ edges, we prove an exact formula for matching subdivisions of arbitrary size $r$, $0\leq r\leq p-2$, namely $\dim(G_r)=p+q-r-4$. Several consequences are included to illustrate how a small matching subdivision can reduce the localization cost of the original zero-divisor network.
\end{abstract}
\par\vspace{2mm}

\baselineskip=0.20in

\section{Introduction}

Let $R$ be a finite commutative ring with unity. The set of zero divisors of $R$ is denoted by $Z(R)$, and $Z(R)^{*}=Z(R)\setminus\{0\}$ denotes the set of non-zero zero divisors. For a positive integer $n$, the ring of integers modulo $n$ is denoted by $\mathbb Z_n$. Beck \cite{beck} initiated the study of zero-divisor graphs of commutative rings, with motivation coming from colorability and the interaction between graph-theoretic and algebraic structure. Anderson and Livingston \cite{anderson} then introduced the now-standard version of the zero-divisor graph. In this paper we follow their convention: the graph $\Gamma(R)$ has vertex set $Z(R)^{*}$, and two distinct vertices $a$ and $b$ are adjacent whenever $ab=0$ in $R$.

For graph-theoretic terminology and algebraic notation not explicitly defined here, we refer to the standard references \cite{gallian}. All graphs considered in the paper are finite, simple, connected, and undirected.

Let $G=(V,E)$ be a connected graph, and let $W=\{w_1,w_2,\ldots,w_k\}$ be an ordered subset of $V(G)$. The metric representation, or metric code, of $v\in V(G)$ with respect to $W$ is
$$
\delta(v\mid W)=(d(v,w_1),d(v,w_2),\ldots,d(v,w_k)).
$$
The set $W$ is called a resolving set if distinct vertices of $G$ have distinct metric representations with respect to $W$. The metric dimension of $G$, denoted by $\dim(G)$, is the minimum cardinality of a resolving set of $G$ \cite{melter,slater}. A resolving set of cardinality $\dim(G)$ is called a metric basis.

Pirzada and Raja \cite{pirzada} initiated the systematic study of the metric dimension of zero-divisor graphs. Later, Pirzada and Imran \cite{pirzada2} studied compressed zero-divisor graphs, and Nithya and Prisci \cite{nitya} considered extended zero-divisor graphs. Several other variants and related metric parameters were subsequently investigated, see \cite{reza,aijaz,ou,ali,bilal} and the references therein.

Throughout the paper, $p$ and $q$ are distinct odd primes with $q>p$. For $n=pq$, the non-zero zero divisors of $\mathbb Z_{pq}$ split into two natural classes:
$$
S_p=\{p_i:1\leq i\leq q-1\},\qquad S_q=\{q_j:1\leq j\leq p-1\}.
$$
Here $S_p$ represents the non-zero multiples of $p$, while $S_q$ represents the non-zero multiples of $q$. The zero-divisor graph $\Gamma(\mathbb Z_{pq})$ is isomorphic to the complete bipartite graph $K_{p-1,q-1}$, with bipartition $S_q\cup S_p$.

The barycentric subdivision of $\Gamma(\mathbb Z_{pq})$ is denoted by $BS(\Gamma(\mathbb Z_{pq}))$. For each edge $q_jp_i$ of $\Gamma(\mathbb Z_{pq})$, let $u_i^j$ be the inserted subdivision vertex. Thus
$$
U^j=\{u_i^j:1\leq i\leq q-1\},\qquad U=\bigcup_{j=1}^{p-1}U^j,
$$
and
$$
V(BS(\Gamma(\mathbb Z_{pq})))=S_q\cup S_p\cup U.
$$
This notation will be used throughout the paper. We also use the closed neighborhood $N[x]=N(x)\cup\{x\}$.

\begin{figure}[H]
	\centering
	\resizebox{0.45\textwidth}{!}{
		\begin{tikzpicture}[scale=1,every node/.style={font=\scriptsize}]
			\node[circle,draw,fill=blue!20] (q1) at (-5,3) {$q_1$};
			\node[circle,draw,fill=blue!20] (q2) at (-5,-3) {$q_2$};
			\node[circle,draw,fill=blue!20] (q3) at (5,3) {$q_3$};
			\node[circle,draw,fill=blue!20] (q4) at (5,-3) {$q_4$};
			
			\foreach \i/\y in {1/4.5,2/3.5,3/2.5,4/1.5,5/0.5,6/-0.5,7/-1.5,8/-2.5,9/-3.5,10/-4.5}{
				\node[circle,draw,fill=red!30] (p\i) at (0,\y) {$p_{\i}$};
			}
			
			\foreach \j/\y in {1/4.3,2/3.7,3/3.1,4/2.5,5/1.9,6/1.3,7/0.7,8/0.1,9/-0.5,10/-1.1}{
				\draw[thick] (q1) -- (-3.2,\y) -- (p\j);
			}
			\foreach \j/\y in {1/1.1,2/0.5,3/-0.1,4/-0.7,5/-1.3,6/-1.9,7/-2.5,8/-3.1,9/-3.7,10/-4.3}{
				\draw[thick] (q2) -- (-1.8,\y) -- (p\j);
			}
			\foreach \j/\y in {1/4.3,2/3.7,3/3.1,4/2.5,5/1.9,6/1.3,7/0.7,8/0.1,9/-0.5,10/-1.1}{
				\draw[thick] (q3) -- (1.8,\y) -- (p\j);
			}
			\foreach \j/\y in {1/1.1,2/0.5,3/-0.1,4/-0.7,5/-1.3,6/-1.9,7/-2.5,8/-3.1,9/-3.7,10/-4.3}{
				\draw[thick] (q4) -- (3.2,\y) -- (p\j);
			}
			
			\foreach \j/\y in {1/4.3,2/3.7,3/3.1,4/2.5,5/1.9,6/1.3,7/0.7,8/0.1,9/-0.5,10/-1.1}{
				\node[circle,draw,fill=yellow!35,inner sep=1.5pt] (m1\j) at (-3.2,\y) {};
			}
			\foreach \j/\y in {1/1.1,2/0.5,3/-0.1,4/-0.7,5/-1.3,6/-1.9,7/-2.5,8/-3.1,9/-3.7,10/-4.3}{
				\node[circle,draw,fill=yellow!35,inner sep=1.5pt] (m2\j) at (-1.8,\y) {};
			}
			\foreach \j/\y in {1/4.3,2/3.7,3/3.1,4/2.5,5/1.9,6/1.3,7/0.7,8/0.1,9/-0.5,10/-1.1}{
				\node[circle,draw,fill=yellow!35,inner sep=1.5pt] (m3\j) at (1.8,\y) {};
			}
			\foreach \j/\y in {1/1.1,2/0.5,3/-0.1,4/-0.7,5/-1.3,6/-1.9,7/-2.5,8/-3.1,9/-3.7,10/-4.3}{
				\node[circle,draw,fill=yellow!35,inner sep=1.5pt] (m4\j) at (3.2,\y) {};
			}
		\end{tikzpicture}
	}
	\caption{Structure of $BS(\Gamma(\mathbb Z_{55}))$.}
	\label{fig:bs-z55}
\end{figure}
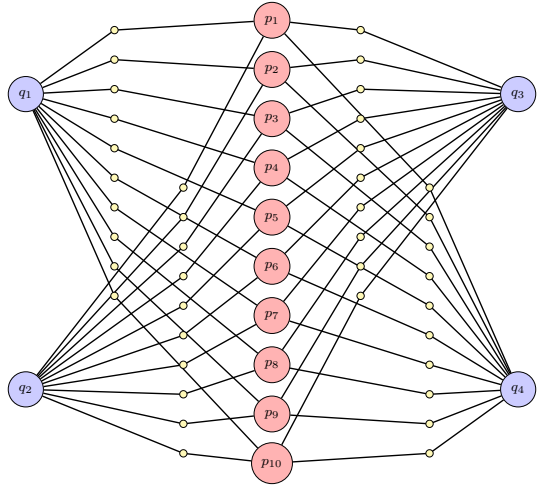

\noindent The next observation records the distance table that will be used repeatedly to compare metric representations.
\begin{remark}\label{remark1}
For $1\leq i,j\leq q-1$ and $1\leq m,n\leq p-1$, the following distance and neighborhood relations hold in $BS(\Gamma(\mathbb Z_{pq}))$:
$$
N[p_i]=\{p_i\}\cup\{u_i^j:1\leq j\leq p-1\},\qquad
N[q_j]=\{q_j\}\cup\{u_i^j:1\leq i\leq q-1\},
$$
$$
N[u_i^j]=\{p_i,u_i^j,q_j\}.
$$
Moreover,
$$
d(p_i,p_k)=4\quad (i\neq k),\qquad d(q_j,q_m)=4\quad (j\neq m),\qquad d(p_i,q_j)=2,
$$
$$
d(p_i,u_k^j)=
\begin{cases}
1, & i=k,\\
3, & i\neq k,
\end{cases}
\qquad
 d(q_j,u_i^m)=
\begin{cases}
1, & j=m,\\
3, & j\neq m,
\end{cases}
$$
and
$$
d(u_i^j,u_k^m)=
\begin{cases}
2, & i=k,\; j\neq m,\\
2, & i\neq k,\; j=m,\\
4, & i\neq k,\; j\neq m.
\end{cases}
$$
\end{remark}

\noindent The following remark explains why vertices in $S_p$ are hard to distinguish unless a landmark lies in one of their closed neighborhoods.
\begin{remark}\label{remar2}
Let $1\leq i\neq j\leq q-1$. If $x\in V(BS(\Gamma(\mathbb Z_{pq})))\setminus\{p_i,p_j\}$ satisfies $d(p_i,x)\neq 1$ and $d(p_j,x)\neq 1$, then $d(p_i,x)=d(p_j,x)$. Hence two vertices of $S_p$ can be separated only by a landmark lying in exactly one of their closed neighborhoods.
\end{remark}

\noindent We recall the known metric-dimension values for the full barycentric subdivision, since they give the cardinality target for the basis characterization below.
\begin{proposition}\rm \cite{77}\label{ET}
Let $p$ and $q$ be two distinct primes with $q>p$, and let $R=\mathbb Z_{pq}$. Then:
\begin{enumerate}[label=$\mathrm{(\arabic*)}$]
\item if $q\geq 2p-1$, then $\dim(BS(\Gamma(R)))=q-2$;
\item if $q=2p-3$, then $\dim(BS(\Gamma(R)))=q-1$;
\item if $p+1<q<2p-3$, then $\dim(BS(\Gamma(R)))>q-2$.
\end{enumerate}
\end{proposition}

The next sections refine the first part of Proposition \ref{ET} by describing exactly which subsets of cardinality $q-2$ are metric bases. We then compare this barycentric model with partial matching subdivisions of $\Gamma(\mathbb Z_{pq})$.

\medskip
The article is organized as: Section \ref{sec-2} gives the characterization of metric bases in $BS(\Gamma(\mathbb Z_{pq}))$ and related results.  Section \ref{sec-m-subdivision} gives the metric dimension of $M$-subdivisions of $\Gamma(\mathbb Z_{pq})$. Section \ref{section 4} gives the structure implications of $\Gamma(\mathbb Z_{pq})$ and its applications. We end up article with conclusion in Section \ref{sect con}.

\section{Characterization of metric bases in $BS(\Gamma(\mathbb Z_{pq}))$}\label{sec-2}

Throughout this section, assume $q\geq 2p-1$. The vertex partition
$$
V(BS(\Gamma(\mathbb Z_{pq})))=S_q\cup S_p\cup U
$$
is fixed.

\noindent The first counting step shows that a set of size $q-2$ must miss at least one closed neighborhood based at a vertex of $S_p$.
\begin{lemma}\rm \cite{fault}\label{lemma1}
Let $W_b\subseteq V(BS(\Gamma(\mathbb Z_{pq})))$ with $|W_b|=q-2$. Then there exists at least one vertex $p_i\in S_p$ such that $N[p_i]\cap W_b=\emptyset$.
\end{lemma}

\begin{proof}
The family $\{N[p_i]:1\leq i\leq q-1\}$ is pairwise disjoint inside $BS(\Gamma(\mathbb Z_{pq}))$, because $N[p_i]=\{p_i\}\cup\{u_i^j:1\leq j\leq p-1\}$ and the index $i$ is fixed throughout this set. Since there are $q-1$ such pairwise disjoint closed neighborhoods but only $q-2$ vertices in $W_b$, the pigeonhole principle gives an index $i$ for which $N[p_i]\cap W_b=\emptyset$.
\end{proof}

\medskip
\noindent The next lemma shows that each non-missed closed neighborhood must receive exactly one landmark; otherwise two vertices of $S_p$ remain indistinguishable.
\begin{lemma}\label{lemma-mid-set}
Let $W_b\subseteq V(BS(\Gamma(\mathbb Z_{pq})))$ with $|W_b|=q-2$, and let $p_i\in S_p$ satisfy $N[p_i]\cap W_b=\emptyset$. If there exists $p_k\in S_p\setminus\{p_i\}$ such that $|N[p_k]\cap W_b|\neq 1$, then there exist distinct vertices $p_a,p_b\in S_p\setminus W_b$ such that $\delta(p_a\mid W_b)=\delta(p_b\mid W_b)$.
\end{lemma}

\begin{proof}
By Lemma \ref{lemma1}, at least one such $p_i$ exists. If $|N[p_k]\cap W_b|=0$ for some $k\neq i$, then no landmark of $W_b$ is at distance $1$ from either $p_i$ or $p_k$. By Remark \ref{remar2}, every landmark has the same distance from $p_i$ and $p_k$. Therefore $\delta(p_i\mid W_b)=\delta(p_k\mid W_b)$.

Now suppose $|N[p_k]\cap W_b|>1$ for some $k\neq i$. Since the sets $N[p_s]$ are pairwise disjoint and $|W_b|=q-2$, placing more than one landmark in $N[p_k]$ forces at least two of the $q-1$ sets $N[p_s]$ to receive no landmark. Thus there is $p_j\in S_p\setminus\{p_i,p_k\}$ with $N[p_j]\cap W_b=\emptyset$. Applying the preceding paragraph to $p_i$ and $p_j$ gives $\delta(p_i\mid W_b)=\delta(p_j\mid W_b)$.
\end{proof}

\medskip
\noindent This lemma identifies the rigid shape forced by the one-landmark-per-neighborhood condition.
\begin{lemma}\label{lemma-IND}
Let $W_b\subseteq V(BS(\Gamma(\mathbb Z_{pq})))$ have $q-2$ vertices. Suppose $N[p_i]\cap W_b=\emptyset$ for one $p_i\in S_p$, and $|N[p_x]\cap W_b|=1$ for every $p_x\in S_p\setminus\{p_i\}$. Then $W_b$ is independent and $W_b\cap S_q=\emptyset$.
\end{lemma}

\begin{proof}
First assume that $W_b\cap S_q\neq\emptyset$. A vertex of $S_q$ is not contained in any of the closed neighborhoods $N[p_x]$. Hence it uses one of the $q-2$ available positions without covering any $N[p_x]$. Since $q-2$ landmarks must cover the $q-2$ neighborhoods $N[p_x]$ with $x\neq i$ exactly once, this is impossible. Equivalently, one additional $N[p_k]$ would be empty, contradicting the assumed equality $|N[p_k]\cap W_b|=1$.

Next assume that $W_b$ is dependent. Since no vertex of $S_q$ lies in $W_b$, dependence can only occur if two selected vertices are adjacent inside the subdivision graph. Such an adjacent pair has the form $p_k u_k^j$ for some $k$ and $j$, because the induced graph on $S_p\cup U$ contains only the edges between $p_k$ and the vertices $u_k^j$. But then both $p_k$ and $u_k^j$ lie in $N[p_k]\cap W_b$, so $|N[p_k]\cap W_b|\geq 2$, again contradicting the hypothesis. Thus $W_b$ is independent and $W_b\cap S_q=\emptyset$.
\end{proof}

\medskip
\noindent The following consequence is immediate from Lemmas \ref{lemma-mid-set} and \ref{lemma-IND}.
\begin{corollary}\label{lemma-ind-and-K}
Let $W_b\subseteq V(BS(\Gamma(\mathbb Z_{pq})))$ with $|W_b|=q-2$. If $W_b$ is dependent or $W_b\cap S_q\neq\emptyset$, then there exist distinct $p_a,p_b\in S_p\setminus W_b$ such that $\delta(p_a\mid W_b)=\delta(p_b\mid W_b)$.
\end{corollary}

\begin{proof}
If the conclusion of Lemma \ref{lemma-IND} fails, then either some neighborhood $N[p_k]$ is empty besides the empty neighborhood guaranteed by Lemma \ref{lemma1}, or one neighborhood receives at least two landmarks. Lemma \ref{lemma-mid-set} gives the desired pair.
\end{proof}

\noindent The next obstruction shows that two poorly represented rows of subdivision vertices force equal metric codes.
\begin{lemma}\label{lemmaD}
Let $W_b\subseteq V(BS(\Gamma(\mathbb Z_{pq})))$ with $|W_b|=q-2$. If there exist distinct indices $m$ and $n$ such that $|U^m\cap W_b|,|U^n\cap W_b|\in\{0,1\}$, then there exist distinct vertices $a,b\in V(BS(\Gamma(\mathbb Z_{pq})))\setminus W_b$ such that $\delta(a\mid W_b)=\delta(b\mid W_b)$.
\end{lemma}

\begin{proof}
By Corollary \ref{lemma-ind-and-K}, we may assume that $W_b$ is independent, $W_b\cap S_q=\emptyset$, and there is a unique $p_i\in S_p$ with $N[p_i]\cap W_b=\emptyset$ while every other $N[p_x]$ contains exactly one landmark. We consider three cases.

If $U^m\cap W_b=\emptyset=U^n\cap W_b$, then every landmark of $W_b$ has the same distance from $q_m$ and $q_n$. Indeed, a landmark in $S_p$ has distance $2$ from both, while a landmark in $U^j$ has distance $1$ from $q_j$ if $j=m$ or $j=n$ and distance $3$ otherwise. Since there is no landmark in $U^m\cup U^n$, all coordinates agree. Hence $\delta(q_m\mid W_b)=\delta(q_n\mid W_b)$.

Suppose $U^m\cap W_b=\emptyset$ and $U^n\cap W_b=\{u_r^n\}$. Let $p_i$ be the unique vertex whose closed neighborhood is missed by $W_b$. Neither \( p_i \) nor any \( u_i^j \) or \( u_r^m \) is included in \( W_b \). The two vertices $u_r^m$ and $u_i^n$ are outside $W_b$. Since $p_i$ and $p_r$ $\notin$ $W_b$, a landmark in $S_p$ has distance $3$ from both. A landmark $u_s^t$ with $s\neq i$ and $t\neq m$ has distance $2$ from both if $t=n$ and $s=r$ and distance $4$ from both otherwise, with the single landmark $u_r^n$ giving distance $2$ to each. Thus $\delta(u_r^m\mid W_b)=\delta(u_i^n\mid W_b)$.

Finally, suppose $U^m\cap W_b=\{u_k^m\}$ and $U^n\cap W_b=\{u_j^n\}$. Since the closed neighborhoods $N[p_k]$ and $N[p_j]$ each contain exactly one landmark, the vertices $u_j^m$ and $u_k^n$ are not in $W_b$. Checking the distances in Remark \ref{remark1}, every landmark outside $U^m\cup U^n$ gives equal coordinates to $u_j^m$ and $u_k^n$, while $u_k^m$ and $u_j^n$ both give coordinate $2$ to each of them. Hence $\delta(u_j^m\mid W_b)=\delta(u_k^n\mid W_b)$.
\end{proof}

\medskip
\noindent The following lemma explains why, if no original $S_p$-vertex is selected, then every row $U^k$ must contain enough landmarks.
\begin{lemma}\label{lemma-sp-empty-row}
Let $W_b\subseteq V(BS(\Gamma(\mathbb Z_{pq})))$ with $|W_b|=q-2$ and $q\geq 2p-1$. If $S_p\cap W_b=\emptyset$ and there exists $k$ such that $|U^k\cap W_b|\leq 1$, then there exist distinct vertices $a,b\in V(BS(\Gamma(\mathbb Z_{pq})))\setminus W_b$ such that $\delta(a\mid W_b)=\delta(b\mid W_b)$.
\end{lemma}

\begin{proof}
If $W_b$ is dependent or contains a vertex of $S_q$, the assertion follows from Corollary \ref{lemma-ind-and-K}. Hence we may assume that $W_b$ is independent, $W_b\cap S_q=\emptyset$, and, by Lemma \ref{lemma-mid-set}, there is a unique vertex $p_i\in S_p$ such that $N[p_i]\cap W_b=\emptyset$, while every other set $N[p_x]$ contains exactly one vertex of $W_b$. Since $S_p\cap W_b=\emptyset$, all landmarks lie in $U$.

Suppose first that $U^k\cap W_b=\emptyset$. For every landmark $u_s^t\in W_b$, we have $s\neq i$ and $t\neq k$. Therefore
$$
d(p_i,u_s^t)=3=d(q_k,u_s^t).
$$
Thus $\delta(p_i\mid W_b)=\delta(q_k\mid W_b)$.

Now suppose that $U^k\cap W_b=\{u_s^k\}$. Necessarily $s\neq i$, because no landmark is contained in $N[p_i]$. We compare $q_k$ and $p_s$. For the unique landmark $u_s^k$, both distances are equal to $1$. For any other landmark $u_a^t\in W_b\setminus\{u_s^k\}$, the equality $a\neq s$ follows from the fact that $N[p_s]\cap W_b=\{u_s^k\}$, and $t\neq k$ follows from $U^k\cap W_b=\{u_s^k\}$. Hence
$$
d(q_k,u_a^t)=3=d(p_s,u_a^t).
$$
Consequently $\delta(q_k\mid W_b)=\delta(p_s\mid W_b)$. In both cases two vertices outside $W_b$ have the same metric representation, as required.
\end{proof}

\medskip
\noindent We can now state the full criterion: a basis is exactly a balanced selection across the $S_p$-neighborhoods and the subdivision rows.
\begin{theorem}\label{thm-barycentric-basis}
Let $p$ and $q$ be distinct odd primes with $q\geq 2p-1$. Let $W_b\subseteq V(BS(\Gamma(\mathbb Z_{pq})))$ with $|W_b|=q-2$. Then $W_b$ is a metric basis of $BS(\Gamma(\mathbb Z_{pq}))$ if and only if all of the following conditions hold:
\begin{enumerate}[label=$\mathrm{(\arabic*)}$]
\item there exists a unique $p_i\in S_p$ such that $N[p_i]\cap W_b=\emptyset$, and $|N[p_k]\cap W_b|=1$ for every $p_k\in S_p\setminus\{p_i\}$;
\item $|U^j\cap W_b|\geq 2$ for every $j$, except possibly for one index $j$;
\item either $S_p\cap W_b\neq\emptyset$, or $|U^j\cap W_b|\geq 2$ for every $j$.
\end{enumerate}
\end{theorem}

\begin{proof}
Let $W_b$ be a metric basis. Since $q\geq 2p-1$, Proposition \ref{ET} gives $|W_b|=q-2$. Lemma \ref{lemma1} gives a vertex $p_i$ whose closed neighborhood is missed. If some other closed neighborhood $N[p_k]$ failed to contain exactly one landmark, Lemma \ref{lemma-mid-set} would produce two vertices of $S_p$ with identical metric representations, contradicting the resolving property. Hence condition $\mathrm{(1)}$ holds. Corollary \ref{lemma-ind-and-K} shows that $W_b$ is independent and contains no vertex of $S_q$. Lemma \ref{lemmaD} forces condition $\mathrm{(2)}$, and Lemma \ref{lemma-sp-empty-row} forces condition $\mathrm{(3)}$.

Conversely, suppose $W_b$ satisfies $\mathrm{(1)}$--$\mathrm{(3)}$. We prove that it resolves the graph. Assume that two distinct vertices $e_1,e_2\in V(BS(\Gamma(\mathbb Z_{pq})))\setminus W_b$ have the same metric representation with respect to $W_b$. We separate the possible positions of $e_1$ and $e_2$.

If $e_1,e_2\in S_p$, then condition $\mathrm{(1)}$ gives a landmark in exactly one of $N[e_1]$ and $N[e_2]$, unless both are the unique missed vertex, which is impossible because they are distinct. That landmark is at distance $1$ from exactly one of them, contradicting equality of metric codes.

If $e_1=u_{n_1}^{m_1}$ and $e_2=u_{n_2}^{m_2}$ lie in $U$, first suppose $m_1=m_2$ and $n_1\neq n_2$. By condition $\mathrm{(1)}$, a landmark lies in exactly one of $N[p_{n_1}]$ and $N[p_{n_2}]$, unless one of these neighborhoods is the unique missed one. In either case, some landmark distinguishes $u_{n_1}^{m_1}$ and $u_{n_2}^{m_2}$ by the distance formulas in Remark \ref{remark1}. If $m_1\neq m_2$, condition $\mathrm{(2)}$ gives two landmarks in one of the rows $U^{m_1}$ or $U^{m_2}$, except possibly for one row. Two distinct landmarks in a row cannot have the same pair of distances to $u_{n_1}^{m_1}$ and $u_{n_2}^{m_2}$ unless the two vertices coincide, which they do not. Hence two subdivision vertices are always separated.

If $e_1,e_2\in S_q$, say $e_1=q_{m_1}$ and $e_2=q_{m_2}$ with $m_1\neq m_2$, then condition $\mathrm{(2)}$ gives a landmark in $U^{m_1}$ or $U^{m_2}$. Such a landmark is at distance $1$ from one of the two vertices and distance $3$ from the other.

Let $e_1\in S_p$ and $e_2\in U$, say $e_1=p_i$ and $e_2=u_j^m$. If $S_p\cap W_b\neq\emptyset$, choose $p_k\in S_p\cap W_b$. Then $d(p_k,p_i)=0$ if $k=i$ and $4$ otherwise, whereas $d(p_k,u_j^m)$ is $1$ or $3$; hence the two vertices are distinguished. If $S_p\cap W_b=\emptyset$, condition $\mathrm{(3)}$ gives two distinct landmarks in $U^m$. At most one of these landmarks can have the same $S_p$-index as $p_i$, so one of them gives different distances to $p_i$ and $u_j^m$.

Let $e_1\in S_p$ and $e_2\in S_q$. If $S_p\cap W_b\neq\emptyset$, a landmark $p_k\in S_p\cap W_b$ gives distance $4$ or $0$ to $e_1$ and distance $2$ to $e_2$. If $S_p\cap W_b=\emptyset$, condition $\mathrm{(3)}$ gives two landmarks in the row corresponding to $e_2$. Since they have distinct $S_p$-indices, at least one is not adjacent to $e_1$; the two distance coordinates are then $3$ and $1$.

Finally, let $e_1\in U$ and $e_2\in S_q$, say $e_1=u_i^m$ and $e_2=q_t$. If $m=t$, any landmark in $S_p\cap W_b$ separates them; if $S_p\cap W_b=\emptyset$, condition $\mathrm{(3)}$ supplies a second landmark in $U^m$ distinct from $u_i^m$, giving distance $2$ to $e_1$ and $1$ to $e_2$. If $m\neq t$, any landmark in $U^m\cap W_b$ has distance $2$ or $0$ from $e_1$ and distance $3$ from $e_2$. Thus every possible pair is separated. Therefore $W_b$ is a resolving set. Since $|W_b|=q-2=\dim(BS(\Gamma(\mathbb Z_{pq})))$, it is a metric basis.
\end{proof}

\medskip
\noindent The characterization has the following immediate structural meaning for every minimum resolving set.
\begin{corollary}\label{cor-basis-independent}
For $q\geq 2p-1$, every metric basis of $BS(\Gamma(\mathbb Z_{pq}))$ is an independent subset of $S_p\cup U$ and contains no vertex of $S_q$.
\end{corollary}

\begin{proof}
This is exactly the structural consequence obtained from condition $\mathrm{(1)}$ and Lemma \ref{lemma-IND} in the proof of Theorem \ref{thm-barycentric-basis}.
\end{proof}

\noindent The next projection result records that a metric basis chooses one representative over each required $S_p$-fiber.
\begin{proposition}\label{prop-projection-basis}
Let $W_b$ be a metric basis of $BS(\Gamma(\mathbb Z_{pq}))$ for $q\geq 2p-1$. Define the projection $\pi_p:S_p\cup U\to S_p$ by $\pi_p(p_i)=p_i$ and $\pi_p(u_i^j)=p_i$. Then $\pi_p$ maps $W_b$ bijectively onto $S_p\setminus\{p_i\}$, where $p_i$ is the unique vertex satisfying $N[p_i]\cap W_b=\emptyset$.
\end{proposition}

\begin{proof}
By Theorem \ref{thm-barycentric-basis}, each closed neighborhood $N[p_k]$ with $k\neq i$ contains exactly one landmark of $W_b$, and $N[p_i]$ contains none. Since the fibers of $\pi_p$ are precisely the sets $N[p_k]\cap(S_p\cup U)$, the image of $W_b$ is $S_p\setminus\{p_i\}$. The cardinalities are equal, so the map is bijective on $W_b$.
\end{proof}

\medskip
\noindent The following example illustrates the three conditions of Theorem \ref{thm-barycentric-basis} in the smallest range displayed in Figure \ref{fig:bs-z55}.
\begin{example}\label{ex-barycentric-z55}
Let $p=5$ and $q=11$. Then $q=2p+1\geq 2p-1$, and $\dim(BS(\Gamma(\mathbb Z_{55})))=9$. The set
$$
W_b=\{p_1,u_2^1,u_3^1,u_4^2,u_5^2,u_6^3,u_7^3,u_8^4,u_9^4\}
$$
has cardinality $9$. It meets $N[p_k]$ exactly once for $1\leq k\leq 9$ and misses $N[p_{10}]$. Moreover, each of $U^1,U^2,U^3,U^4$ contains two landmarks, and $W_b\cap S_p=\{p_1\}$. By Theorem \ref{thm-barycentric-basis}, $W_b$ is a metric basis of $BS(\Gamma(\mathbb Z_{55}))$.
\end{example}

\section{Metric dimension of $M$-subdivisions of $\Gamma(\mathbb Z_{pq})$}\label{sec-m-subdivision}

In this section, $G=\Gamma(\mathbb Z_{pq})\cong K_{p-1,q-1}$ has bipartition $S_q\cup S_p$. An $M$-subdivision of $G$ is obtained by subdividing selected edges of $G$. If the selected edges are pairwise disjoint, we call the resulting graph a matching subdivision. For a subdivided edge $q_ip_j$, the inserted vertex is denoted by $m_{ij}$.

For each $q_i\in S_q$, define
$$
M_i=\{m_{ij}:q_ip_j\text{ is subdivided}\},\qquad c_i=|M_i|.
$$
When a fixed number of edges is subdivided, the sets $M_i$ are classified as follows:
\begin{enumerate}[label=$\mathrm{(\arabic*)}$]
\item $M_i$ is a null partition if $c_i=0$;
\item $M_i$ is a unique partition if $c_i=1$;
\item $M_i$ is a good partition if $c_i\geq 2$.
\end{enumerate}
Let
$$
e_1=|\{i:c_i=0\}|,
\qquad o=|\{i:c_i=1\}|,
\qquad g=|\{i:c_i\geq 2\}|.
$$
Then $e_1+o+g=p-1$.

\medskip
\noindent The first subdivision count compares null and good partitions when only $p-3$ edges are subdivided.
\begin{lemma}\label{lemma-null-count-p3}
If exactly $p-3$ edges are subdivided, then $e_1\geq g+2$.
\end{lemma}

\begin{proof}
Since exactly $p-3$ edges are subdivided,
$$
\sum_{i=1}^{p-1}c_i=p-3.
$$
Every unique partition contributes $1$, while every good partition contributes at least $2$. Hence
$$
p-3=\sum_{i=1}^{p-1}c_i\geq o+2g.
$$
Using $o=p-1-e_1-g$, we obtain
$$
p-3\geq p-1-e_1-g+2g=p-1-e_1+g,
$$
so $e_1\geq g+2$.
\end{proof}

\medskip
\noindent Null partitions create twin vertices in $S_q$, and the next lemma converts that twin structure into a lower bound for resolving sets.
\begin{lemma}\label{lemma3.1}
Let $W_m$ be a resolving set of an $M$-subdivision of $\Gamma(\mathbb Z_{pq})$. If $e_1$ denotes the number of null partitions, then $|W_m\cap S_q|\geq e_1-1$.
\end{lemma}

\begin{proof}
Let $M_{i_1},M_{i_2},\ldots,M_{i_{e_1}}$ be the null partitions. Since no edge incident with $q_{i_s}$ is subdivided, the vertices $q_{i_1},q_{i_2},\ldots,q_{i_{e_1}}$ have the same open neighborhood, namely $S_p$. Therefore they form a twin class. If a resolving set omitted two vertices from this class, those two omitted vertices would have the same distance to every landmark: distance $1$ to landmarks in $S_p$, distance $2$ to landmarks in $S_q$, and distance $2$ to subdivision landmarks. Thus all but at most one of the null-partition vertices must belong to $W_m$, giving $|W_m\cap S_q|\geq e_1-1$.
\end{proof}

\medskip
\noindent The preceding twin bound also forces many vertices on the $S_p$-side to remain outside a small resolving set.
\begin{corollary}\label{cor-sp-outside}
If $|W_m|=q-2$, then at least $e_1$ vertices of $S_p$ lie outside $W_m$.
\end{corollary}

\begin{proof}
By Lemma \ref{lemma3.1}, at least $e_1-1$ vertices of $W_m$ lie in $S_q$. Hence at most
$$
(q-2)-(e_1-1)=q-e_1-1
$$
vertices of $S_p$ can belong to $W_m$. Since $|S_p|=q-1$, at least
$$
(q-1)-(q-e_1-1)=e_1
$$
vertices of $S_p$ lie outside $W_m$.
\end{proof}

\medskip
\noindent The next lemma isolates a large subset of $S_p$ that is invisible to the selected non-$S_q$ landmarks.
\begin{lemma}\label{lemma-3.2}
Let $W_m$ be a resolving set of an $M$-subdivision of $\Gamma(\mathbb Z_{pq})$ with $|W_m|=q-2$. Suppose
$$
|W_m\cap S_q|=(e_1-1)+a.
$$
Then there exists $S_1\subseteq S_p$ such that
$$
|S_1|\geq e_1+a,
$$
$S_1\cap W_m=\emptyset$, and no vertex of $S_1$ is adjacent to a subdivision vertex belonging to $W_m\setminus S_q$.
\end{lemma}

\begin{proof}
Since $|W_m|=q-2$ and $|W_m\cap S_q|=(e_1-1)+a$, we have
$$
|W_m\setminus S_q|=q-e_1-a-1.
$$
Every vertex of $W_m\setminus S_q$ accounts for at most one vertex of $S_p$: if it is a vertex of $S_p$, it accounts for itself; if it is a subdivision vertex, it is adjacent to exactly one vertex of $S_p$. Hence at most $q-e_1-a-1$ vertices of $S_p$ are either selected in $W_m$ or adjacent to a selected subdivision vertex. Since $|S_p|=q-1$, at least
$$
(q-1)-(q-e_1-a-1)=e_1+a
$$
vertices remain. Taking $S_1$ to be the set of all remaining vertices gives the result.
\end{proof}

\medskip
\noindent This lemma prevents too many vertices of the invisible set from being completely unaffected by subdivision.
\begin{lemma}\label{lemma-3.3}
Let $S_1$ be the set obtained in Lemma \ref{lemma-3.2}. Then at most one vertex of $S_1$ is not incident with any subdivided edge.
\end{lemma}

\begin{proof}
If two distinct vertices $s_a,s_b\in S_1$ were not incident with any subdivided edge, then both would have open neighborhood $S_q$ in the subdivision graph. Hence they would be twins. Since $S_1\cap W_m=\emptyset$, neither belongs to $W_m$. Every landmark in $S_q$ has distance $1$ from both, every landmark in $S_p$ has distance $2$ from both, and every subdivision landmark has distance $2$ from both. Thus $\delta(s_a\mid W_m)=\delta(s_b\mid W_m)$, contradicting the resolving property.
\end{proof}

\medskip
\noindent The next estimate measures how much distinguishing power a good or unique partition can contribute.
\begin{lemma}\label{lemma-3.4}
Let $S_1$ be as in Lemma \ref{lemma-3.2}. Each good set $V_i=M_i\cup\{q_i\}$ can distinguish at most one vertex of $S_1$. A unique set $V_j=M_j\cup\{q_j\}$ can distinguish a vertex of $S_1$ only if $q_j\in W_m$.
\end{lemma}

\begin{proof}
Fix a good partition $M_i$. If two vertices $s_a,s_b\in S_1$ are both incident with subdivision vertices from $M_i$, then those subdivision vertices do not lie in $W_m$ by the definition of $S_1$. For every landmark $x\in V_i\cap W_m$, the distances from $x$ to $s_a$ and $s_b$ are equal, unless exactly one of $s_a,s_b$ is the $S_p$-neighbor of a selected subdivision vertex. This exceptional situation can occur for at most one vertex of $S_1$. Thus one good set can single out at most one vertex of $S_1$.

If $M_j$ is unique, say $M_j=\{m_{jt}\}$, then the only possible landmark in $V_j$ that can separate the unique $S_p$-neighbor $p_t$ from the other vertices of $S_1$ is $q_j$. The subdivision vertex $m_{jt}$ is not usable for this purpose when $p_t\in S_1$, because by definition no vertex of $S_1$ is adjacent to a subdivision landmark in $W_m\setminus S_q$. Therefore a unique set contributes only when its corresponding vertex $q_j$ lies in $W_m$.
\end{proof}

\medskip
\noindent Combining the previous lemma over all non-null partitions gives the following global counting bound.
\begin{corollary}\label{coro-3.2}
If $|W_m\cap S_q|=(e_1-1)+a$, then the number of vertices of $S_1$ that can be distinguished through good and unique sets is at most $g+a$.
\end{corollary}

\begin{proof}
There are $g$ good sets, and each contributes to at most one vertex of $S_1$ by Lemma \ref{lemma-3.4}. Among the $(e_1-1)+a$ vertices of $S_q\cap W_m$, at least $e_1-1$ are needed for null partitions. Therefore at most $a$ vertices of $S_q\cap W_m$ can correspond to non-null unique partitions. Lemma \ref{lemma-3.4} gives the bound $g+a$.
\end{proof}

\medskip
\noindent The following lower bound shows that subdividing only $p-3$ edges cannot achieve a resolving set of size $q-2$.
\begin{theorem}\label{thm-p3-lower}
Let $G$ be obtained from $\Gamma(\mathbb Z_{pq})$, where $q>p$, by subdividing exactly $p-3$ edges. Then $\dim(G)\geq q-1$.
\end{theorem}

\begin{proof}
Assume that $W_m$ is a resolving set of $G$ with $|W_m|=q-2$. Write
$$
|W_m\cap S_q|=(e_1-1)+a.
$$
By Lemma \ref{lemma-3.2}, there is a set $S_1\subseteq S_p$ with $|S_1|\geq e_1+a$, disjoint from $W_m$, and not adjacent to any subdivision landmark in $W_m\setminus S_q$. Lemma \ref{lemma-3.3} shows that at most one vertex of $S_1$ is not incident with a subdivided edge. Hence at least $e_1+a-1$ vertices of $S_1$ must be distinguished through good or unique sets. By Corollary \ref{coro-3.2}, at most $g+a$ such vertices can be distinguished. Lemma \ref{lemma-null-count-p3} gives $e_1\geq g+2$, and therefore
$$
e_1+a-1\geq g+a+1>g+a.
$$
This is impossible. Hence no resolving set of cardinality $q-2$ exists, and $\dim(G)\geq q-1$.
\end{proof}

\medskip
\noindent When one more edge is subdivided, the null-good partition inequality weakens by exactly one unit.
\begin{lemma}\label{lemma-null-count-p2}
If exactly $p-2$ edges are subdivided, then $e_1\geq g+1$.
\end{lemma}

\begin{proof}
Now $\sum_{i=1}^{p-1}c_i=p-2$. As before,
$$
p-2=\sum_{i=1}^{p-1}c_i\geq o+2g.
$$
Using $o=p-1-e_1-g$, we get
$$
p-2\geq p-1-e_1+g,
$$
which gives $e_1\geq g+1$.
\end{proof}

\medskip
\noindent This lower bound is the threshold estimate needed for matching subdivisions of size $p-2$.
\begin{theorem}\label{theorem-3.2}
Let $G$ be obtained from $\Gamma(\mathbb Z_{pq})$, where $q>p$, by subdividing exactly $p-2$ edges. Then $\dim(G)\geq q-2$.
\end{theorem}

\begin{proof}
Assume that $W_m$ is a resolving set of $G$ with $|W_m|=q-3$. Write $|W_m\cap S_q|=(e_1-1)+a$. As in Lemma \ref{lemma-3.2}, the number of vertices of $S_p$ neither in $W_m$ nor adjacent to a subdivision vertex of $W_m\setminus S_q$ is at least
$$
(q-1)-\bigl((q-3)-((e_1-1)+a)\bigr)=e_1+a+1.
$$
Let $S_1$ be this set. By Lemma \ref{lemma-3.3}, at least $e_1+a$ vertices of $S_1$ are incident with subdivided edges. By Lemma \ref{lemma-3.4} and Corollary \ref{coro-3.2}, at most $g+a$ of them can be distinguished through good and unique partitions. Since Lemma \ref{lemma-null-count-p2} gives $e_1\geq g+1$, we obtain
$$
e_1+a\geq g+a+1>g+a,
$$
which is impossible. Thus no resolving set of cardinality $q-3$ exists, and $\dim(G)\geq q-2$.
\end{proof}

\medskip
\noindent The next result shows that a matching subdivision of size $p-2$ actually attains the lower bound.
\begin{theorem}\label{theorem-3.3}
Let $G$ be obtained from $\Gamma(\mathbb Z_{pq})$, where $q>p$, by subdividing the edges of a matching of size $p-2$. Then $\dim(G)=q-2$.
\end{theorem}

\begin{proof}
The lower bound is Theorem \ref{theorem-3.2}. Let the subdivided matching edges be
$$
q_1p_1,q_2p_2,\ldots,q_{p-2}p_{p-2},
$$
and let $m_i$ be the subdivision vertex inserted on $q_ip_i$. Put
$$
T_p=\{p_i\in S_p:p_i\text{ is not incident with a subdivided edge}\}.
$$
Then $|T_p|=(q-1)-(p-2)=q-p+1$. Choose $p^*\in T_p$ and define
$$
W_m=\{m_1,m_2,\ldots,m_{p-2}\}\cup (T_p\setminus\{p^*\}).
$$
Thus
$$
|W_m|=(p-2)+(q-p)=q-2.
$$
We verify that $W_m$ resolves $G$. If $q_i$ and $q_j$ are two distinct vertices of $S_q$ with $1\leq i,j\leq p-2$, then $m_i$ separates them because $d(q_i,m_i)=1$ and $d(q_j,m_i)=2$. If one of them is the unmatched vertex $q_{p-1}$, the same argument with $m_i$ still separates the pair. If $p_i$ and $p_j$ are two matched vertices of $S_p$, then $m_i$ separates them, since $d(p_i,m_i)=1$ and $d(p_j,m_i)=2$. If $p_t\in T_p\setminus\{p^*\}$, then $p_t\in W_m$ is separated from every other vertex by its zero coordinate; the remaining vertex $p^*$ is separated from each matched $p_i$ by $m_i$, because $d(p_i,m_i)=1$ and $d(p^*,m_i)=2$.

It remains only to compare vertices from different types. A vertex $q_i$ and an unmatched vertex $p_t\in T_p$ are separated as follows. If $q_i$ is incident with a subdivided matching edge, then $m_i$ gives $d(q_i,m_i)=1$ and $d(p_t,m_i)=2$. If $q_i=q_{p-1}$ is the unmatched vertex of $S_q$, then a landmark $p_s\in T_p\setminus\{p^*\}$ separates the pair, either by a zero coordinate when $s=t$, or by $d(q_i,p_s)=1$ and $d(p_t,p_s)=2$ when $s\neq t$. A vertex $q_i$ and the matched vertex $p_i$ are not separated by the common subdivision landmark $m_i$ alone, since both are adjacent to it; however, a landmark $p_t\in T_p\setminus\{p^*\}$ gives $d(q_i,p_t)=1$ and $d(p_i,p_t)=2$. Finally, every subdivision vertex $m_i$ belongs to $W_m$, so its zero coordinate separates it from each endpoint and from every other subdivision vertex. Thus every two vertices have distinct metric codes with respect to $W_m$. Hence $W_m$ is resolving. Therefore $\dim(G)\leq q-2$, and the theorem follows.
\end{proof}

\medskip
\noindent The following general formula explains how each added matching subdivision edge lowers the metric dimension by one.
\begin{theorem}\label{thm-general-matching}
Let $G_r$ be obtained from $\Gamma(\mathbb Z_{pq})$ by subdividing the edges of a matching of size $r$, where $0\leq r\leq p-2$. Then
$$
\dim(G_r)=p+q-r-4.
$$
\end{theorem}

\begin{proof}
Write $a=p-1$ and $b=q-1$, so $\Gamma(\mathbb Z_{pq})\cong K_{a,b}$. Without loss of generality, suppose that the subdivided matching edges are $q_ip_i$ for $1\leq i\leq r$, and let $m_i$ be the subdivision vertex on $q_ip_i$.

First we construct a resolving set. Define
$$
W=\{q_1,q_2,\ldots,q_{a-1}\}\cup\{p_{r+1},p_{r+2},\ldots,p_{b-1}\}.
$$
The first set has $a-1=p-2$ vertices and the second has $b-r-1=q-r-2$ vertices. Therefore
$$
|W|=p+q-r-4.
$$
We verify that $W$ resolves $G_r$. Two vertices of $S_q$ are separated by a selected $q$-vertex unless one of them is the omitted vertex $q_a$; in that remaining case, a selected vertex $p_{r+1}$, when needed, or one of the selected $q$-vertices separates the pair by the usual $0$ and $2$ coordinates. The vertices $p_{r+1},\ldots,p_{b-1}$ similarly separate all vertices of $S_p$ except the single omitted vertex $p_b$, and any matched vertex $p_i$ is separated from $p_b$ by $q_i$, since $d(q_i,p_i)=2$ and $d(q_i,p_b)=1$. A pair consisting of one vertex from $S_q$ and one from $S_p$ is separated either by a selected vertex of $S_q$ or by a selected vertex of $S_p$; the subdivided pair $q_i,p_i$ is separated by $q_i$, because the corresponding distances are $0$ and $2$. Finally, $m_i$ is separated from $q_i$ and $p_i$ by $q_i$, and if $m_i\neq m_j$, then $q_i$ separates them because $d(q_i,m_i)=1$ while $d(q_i,m_j)=2$. Hence $W$ is resolving, and
$$
\dim(G_r)\leq p+q-r-4.
$$

For the lower bound, let $R$ be any resolving set. The unmatched vertices
$$
Q_0=\{q_{r+1},q_{r+2},\ldots,q_a\}
$$
form a twin class, because they are adjacent to exactly the same vertices of $S_p$ and to no subdivision vertex. Therefore
$$
|R\cap Q_0|\geq |Q_0|-1=a-r-1.
$$
Similarly,
$$
P_0=\{p_{r+1},p_{r+2},\ldots,p_b\}
$$
is a twin class, so
$$
|R\cap P_0|\geq |P_0|-1=b-r-1.
$$
For $1\leq i\leq r$, put $C_i=\{q_i,p_i,m_i\}$. If $R\cap C_i=\emptyset$ and some vertex $q_0\in Q_0$ is not in $R$, then $q_i$ and $q_0$ have the same distance to every landmark outside $C_i$: they are both at distance $1$ from all vertices of $P_0$, both at distance $2$ from all vertices of $Q_0\setminus\{q_0\}$, and have equal distances to all other triples $C_j$ with $j\neq i$. This contradicts the resolving property. Hence, whenever a triple $C_i$ is missed, all vertices of $Q_0$ must be contained in $R$. The same argument with the two partite sets interchanged shows that a missed $C_i$ also forces all vertices of $P_0$ to be contained in $R$.

Moreover, two different triples cannot both be missed. Indeed, if $R\cap C_i=R\cap C_j=\emptyset$ with $i\neq j$, then the vertices $q_i$ and $q_j$ have the same distances to every landmark outside $C_i\cup C_j$, while there are no landmarks inside those two triples. Thus $R$ would not resolve $q_i$ and $q_j$.

Therefore at least $r-1$ of the triples $C_i$ meet $R$. If all $r$ triples meet $R$, then together with the two twin-class bounds we get
$$
|R|\geq (a-r-1)+(b-r-1)+r=a+b-r-2.
$$
If exactly one triple is missed, then all vertices of $Q_0$ and all vertices of $P_0$ are contained in $R$, giving two additional vertices beyond the twin-class lower bounds, and hence
$$
|R|\geq (a-r-1)+(b-r-1)+(r-1)+2=a+b-r-1>a+b-r-2.
$$
In all cases $|R|\geq a+b-r-2=p+q-r-4$. Combining this with the upper bound proves
$$
\dim(G_r)=p+q-r-4.
$$
\end{proof}

\medskip
\noindent The case $r=p-3$ gives the exact value corresponding to the lower bound in Theorem \ref{thm-p3-lower}.
\begin{corollary}\label{cor-p3-exact}
If $G$ is obtained from $\Gamma(\mathbb Z_{pq})$ by subdividing the edges of a matching of size $p-3$, then $\dim(G)=q-1$.
\end{corollary}

\begin{proof}
Put $r=p-3$ in Theorem \ref{thm-general-matching}. Then
$$
\dim(G)=p+q-(p-3)-4=q-1.
$$
\end{proof}

\medskip
\noindent The following numerical example shows the change from $q-1$ to $q-2$ when the matching size increases by one.
\begin{example}\label{ex-matching-z55}
For $p=5$ and $q=11$, the graph $\Gamma(\mathbb Z_{55})$ is $K_{4,10}$. If a matching of size $r=3=p-2$ is subdivided, then Theorem \ref{thm-general-matching} gives
$$
\dim(G_3)=5+11-3-4=9=q-2.
$$
If a matching of size $r=2=p-3$ is subdivided, then
$$
\dim(G_2)=5+11-2-4=10=q-1.
$$
Thus the third matching subdivision edge is exactly the edge that lowers the metric dimension from $q-1$ to $q-2$.
\end{example}

\section{Structural implications and applications}\label{section 4}

The preceding results show that barycentric subdivision and matching subdivision can lead to the same value of metric dimension, but with very different numbers of inserted vertices. In the barycentric subdivision, every edge of $\Gamma(\mathbb Z_{pq})$ is subdivided, producing $(p-1)(q-1)$ new vertices. By contrast, Theorem \ref{theorem-3.3} shows that subdividing a matching of size $p-2$ already gives metric dimension $q-2$.

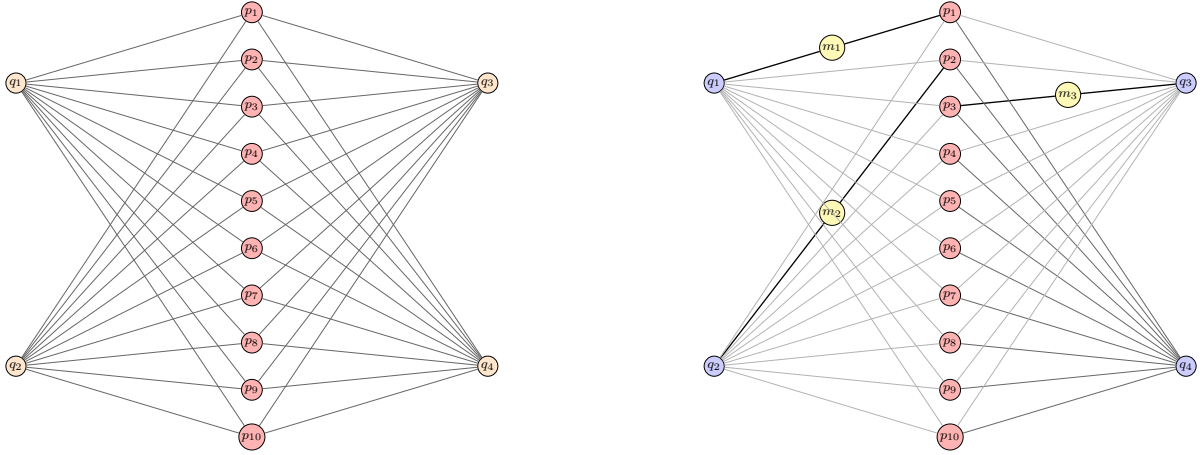
\begin{figure}[H]
\centering
\begin{minipage}{0.42\textwidth}
\centering
\resizebox{\linewidth}{!}{
\begin{tikzpicture}[every node/.style={font=\scriptsize}]
\node[circle,draw,fill=orange!20,inner sep=1pt] (q1) at (-5,3) {$q_1$};
\node[circle,draw,fill=orange!20,inner sep=1pt] (q2) at (-5,-3) {$q_2$};
\node[circle,draw,fill=orange!20,inner sep=1pt] (q3) at (5,3) {$q_3$};
\node[circle,draw,fill=orange!20,inner sep=1pt] (q4) at (5,-3) {$q_4$};
\foreach \i/\y in {1/4.5,2/3.5,3/2.5,4/1.5,5/0.5,6/-0.5,7/-1.5,8/-2.5,9/-3.5,10/-4.5}{
\node[circle,draw,fill=red!30,inner sep=1pt] (p\i) at (0,\y) {$p_{\i}$};
}
\foreach \j in {1,...,10}{
\draw[black!60] (q1)--(p\j);
\draw[black!60] (q2)--(p\j);
\draw[black!60] (q3)--(p\j);
\draw[black!60] (q4)--(p\j);
}
\end{tikzpicture}
}
\end{minipage}
\hfill
\begin{minipage}{0.42\textwidth}
\centering
\resizebox{\linewidth}{!}{
\begin{tikzpicture}[every node/.style={font=\scriptsize}]
\node[circle,draw,fill=blue!20,inner sep=1pt] (q1) at (-5,3) {$q_1$};
\node[circle,draw,fill=blue!20,inner sep=1pt] (q2) at (-5,-3) {$q_2$};
\node[circle,draw,fill=blue!20,inner sep=1pt] (q3) at (5,3) {$q_3$};
\node[circle,draw,fill=blue!20,inner sep=1pt] (q4) at (5,-3) {$q_4$};
\foreach \i/\y in {1/4.5,2/3.5,3/2.5,4/1.5,5/0.5,6/-0.5,7/-1.5,8/-2.5,9/-3.5,10/-4.5}{
\node[circle,draw,fill=red!30,inner sep=1pt] (p\i) at (0,\y) {$p_{\i}$};
}
\node[circle,draw,fill=yellow!35,inner sep=1pt] (m1) at (-2.5,3.75) {$m_1$};
\node[circle,draw,fill=yellow!35,inner sep=1pt] (m2) at (-2.5,0.25) {$m_2$};
\node[circle,draw,fill=yellow!35,inner sep=1pt] (m3) at (2.5,2.75) {$m_3$};
\draw[thick] (q1)--(m1)--(p1);
\draw[thick] (q2)--(m2)--(p2);
\draw[thick] (q3)--(m3)--(p3);
\foreach \j in {1,...,10}{
\ifnum\j=1\else \draw[gray!60] (q1)--(p\j);\fi
\ifnum\j=2\else \draw[gray!60] (q2)--(p\j);\fi
\ifnum\j=3\else \draw[gray!60] (q3)--(p\j);\fi
\draw[black!60] (q4)--(p\j);
}
\end{tikzpicture}
}
\end{minipage}
\caption{The graph $\Gamma(\mathbb Z_{55})$ and a matching subdivision of size $r=p-2$.}
\label{fig:matching-z55}
\end{figure}

\begin{table}[H]
\centering
\caption{Comparison among $\Gamma(\mathbb Z_{pq})$, matching subdivisions, and barycentric subdivisions.}
\label{tab:comparison}
\begin{tabular}{|>{\centering\arraybackslash}p{0.28\textwidth}|>{\centering\arraybackslash}p{0.19\textwidth}|>{\centering\arraybackslash}p{0.23\textwidth}|>{\centering\arraybackslash}p{0.18\textwidth}|}
\hline
Graph & Condition & Number of subdivision vertices & Metric dimension \\
\hline
$\Gamma(\mathbb Z_{pq})$ & $q>p$ & $0$ & $p+q-4$ \\
\hline
Matching subdivision of size $r$ & $0\leq r\leq p-2$ & $r$ & $p+q-r-4$ \\
\hline
Matching subdivision & $r=p-2$ & $p-2$ & $q-2$ \\
\hline
Barycentric subdivision & $q\geq 2p-1$ & $(p-1)(q-1)$ & $q-2$ \\
\hline
Barycentric subdivision & $q=2p-3$ & $(p-1)(q-1)$ & $q-1$ \\
\hline
Barycentric subdivision & $p+1<q<2p-3$ & $(p-1)(q-1)$ & $>q-2$ \\
\hline
\end{tabular}
\end{table}

\noindent The next proposition interprets the exact formula as an efficiency statement for inserted subdivision vertices.
\begin{proposition}\label{prop-efficiency}
Let $G_r$ be the matching subdivision of $\Gamma(\mathbb Z_{pq})$ by a matching of size $r$, where $0\leq r\leq p-2$. The reduction in metric dimension per inserted subdivision vertex is equal to $1$ for every $r\geq 1$.
\end{proposition}

\begin{proof}
The original graph $\Gamma(\mathbb Z_{pq})\cong K_{p-1,q-1}$ has metric dimension $p+q-4$. By Theorem \ref{thm-general-matching}, $\dim(G_r)=p+q-r-4$. Hence the reduction is
$$
(p+q-4)-(p+q-r-4)=r.
$$
Dividing by the $r$ inserted subdivision vertices gives efficiency $1$.
\end{proof}

\noindent The following consequence compares the subdivision cost of the barycentric construction with the matching construction.
\begin{corollary}\label{cor-barycentric-efficiency}
For $q\geq 2p-1$, the barycentric subdivision and the matching subdivision of size $p-2$ have the same metric dimension $q-2$, but the matching subdivision uses fewer inserted vertices. The ratio of the numbers of inserted vertices is
$$
\frac{(p-1)(q-1)}{p-2}.
$$
\end{corollary}

\begin{proof}
The barycentric subdivision inserts one vertex on every edge of $K_{p-1,q-1}$, hence inserts $(p-1)(q-1)$ vertices. The matching subdivision in Theorem \ref{theorem-3.3} inserts only $p-2$ vertices. Both graphs have metric dimension $q-2$, so the displayed ratio measures the difference in subdivision cost.
\end{proof}

\subsection{Network localization}

In network localization problems \cite{kh}, a resolving set represents a collection of landmark nodes used to identify the location of every node by its vector of distances to the landmarks. The metric dimension is therefore the minimum number of landmarks needed for unique localization.

For the zero-divisor graph $\Gamma(\mathbb Z_{pq})\cong K_{p-1,q-1}$, the metric dimension is $p+q-4$ \cite{chart}. The matching subdivision model gives a controlled way to reduce this number. By Theorem \ref{thm-general-matching}, subdividing a matching of size $r$ changes the metric dimension to $p+q-r-4$. Thus every matching subdivision edge contributes exactly one unit of metric-dimension reduction until $r=p-2$. At that point the metric dimension becomes $q-2$, which is the same value obtained by the barycentric subdivision under $q\geq 2p-1$, but it is achieved with only $p-2$ inserted vertices instead of $(p-1)(q-1)$ inserted vertices.

This comparison is useful in localization-based network design. If subdivision vertices model intermediate reference points, relay stations, or auxiliary sensors, then the matching subdivision strategy gives the same resolving strength as barycentric subdivision in the large-prime range, while requiring a much smaller number of inserted nodes.

\noindent The final example translates the formula into a concrete localization-cost comparison.
\begin{example}\label{ex-localization-cost}
For $p=5$ and $q=11$, the original graph $\Gamma(\mathbb Z_{55})$ has metric dimension $5+11-4=12$. A matching subdivision of size $3$ has metric dimension $9$. Hence three inserted reference points reduce the number of landmarks by three. The barycentric subdivision also has metric dimension $9$, but it inserts $(5-1)(11-1)=40$ new vertices. Thus, in this example, matching subdivision achieves the same metric dimension using only $3$ inserted vertices instead of $40$.
\end{example}

\section{Conclusion}\label{sect con}

We have refined the study of metric bases in the barycentric subdivision of the zero-divisor graph $\Gamma(\mathbb Z_{pq})$ by giving a detailed necessary and sufficient condition for subsets of cardinality $q-2$ to be metric bases when $q\geq 2p-1$. The proof shows that such a basis is forced to avoid $S_q$, to select exactly one representative from all but one of the closed neighborhoods $N[p_i]$, and to distribute enough landmarks across the rows $U^j$ of subdivision vertices. For partial subdivisions, we expanded the counting argument based on null, unique, and good partitions. In the matching case, the exact formula $\dim(G_r)=p+q-r-4$ for $0\leq r\leq p-2$ shows that each subdivided matching edge lowers the metric dimension by one. This provides a compact alternative to barycentric subdivision and clarifies why a matching of size $p-2$ is sufficient to obtain the value $q-2$.

\end{document}